\documentclass[a4paper,11pt]{article}
\usepackage[latin1]{inputenc}
\usepackage[swedish,english]{babel}
\usepackage{amsmath,amsthm}
\usepackage{amssymb}
\usepackage{ifpdf}
\usepackage[margin=2.5cm]{geometry}
\ifpdf
  \usepackage[pdftex]{graphicx}
  \usepackage{epstopdf}
\else
  \usepackage[dvips]{graphicx}
\fi
\usepackage[pdftitle={The PDF title},
  pdfauthor={Your Name},
  pdffitwindow=true,
  breaklinks=true,
  colorlinks=true,
  urlcolor=blue,
  linkcolor=red,
  citecolor=red,
  anchorcolor=red]{hyperref}
\usepackage{caption}
\usepackage{subcaption}
\usepackage{accents}
\usepackage[misc]{ifsym}
\newlength{\dhatheight}

\newcommand{\ve}{\varepsilon}
\newcommand{\wt}{\widetilde}
\newcommand{\wh}{\widehat}
\newcommand{\wb}{\overline}
\newcommand{\f}{\frac}
\newcommand{\p}{\partial}
\newcommand{\bs}{\boldsymbol}
\newcommand{\z}{\boldsymbol{z}}
\newcommand{\bz}{\bar{\boldsymbol{z}}}

\newcommand{\be}{\begin{equation}}
\newcommand{\eb}{\end{equation}}

\newcommand{\revii}[1]{{\color{black}#1}}
\newcommand{\reviii}[1]{{\color{black}#1}}
    \theoremstyle{plain}
\newtheorem{theorem}{Theorem}

\theoremstyle{definition}
\newtheorem{definition}{Definition}
\newtheorem{assumption}[definition]{Assumption}

\theoremstyle{remark}
\newtheorem*{remark}{Remark}

\title{An energy-based summation-by-parts finite difference method for the wave equation in second order form}

\author{Siyang Wang\thanks{Department of Mathematics and Mathematical Statistics, Ume{\aa} University, {Ume\aa}, Sweden. Email: siyang.wang@umu.se}        \and
Daniel Appel\"{o}\thanks{Department of Computational Mathematics, Science \& Engineering and Department of Mathematics, Michigan State University, East Lansing, USA. Email: appeloda@msu.edu. 
This work was supported by NSF Grant DMS-1913076. Any opinions, findings, and conclusions or recommendations expressed in this material are those of the authors and do not necessarily reflect the views of the National Science Foundation.} \and
Gunilla Kreiss \thanks{Department of Information Technology, Uppsala University, Uppsala, Sweden. Email: gunilla.kreiss@it.uu.se}}

\begin{document}

\maketitle

\begin{abstract}
\revii{We develop a new finite difference method for the wave equation in second order form. The finite difference operators satisfy a summation-by-parts (SBP) property. With boundary conditions and material interface conditions imposed weakly by the simultaneous-approximation-term (SAT)  method, we derive energy estimates for the semi-discretization. In addition, error estimates are derived by the normal mode analysis. The proposed method is termed as energy-based because of its similarity with the energy-based discontinuous Galerkin method. When imposing the Dirichlet boundary condition and material interface conditions, the traditional SBP-SAT discretization uses a penalty term with a mesh-dependent parameter, which is not needed in our method. }Furthermore, numerical dissipation can be added to the discretization through the boundary and interface conditions.  We present numerical experiments that verify convergence and robustness of the proposed method.
\end{abstract}

\section{Introduction}
Many wave phenomena are governed by second order hyperbolic partial differential equations (PDEs), such as the wave equation, the elastic wave equation and Einstein's equations of general relativity. Also, equations that are formulated in first order form, such as Maxwell's equations, may often be reformulated in second order form \cite{Henshaw2006}. When a second order formulation is available it is typically formulated using fewer variables and fewer derivative operators, which can be exploited in the design of faster numerical methods. Many physical models are derived using Euler-Lagrange formalism starting from an energy, and it is natural to look for numerical methods that track the dynamics of this energy. We have considered such energy-based numerical methods in the context of discontinuous Galerkin discretizations, \cite{Appelo2015,Appelo2018,Appelo2019}. Here we generalize those ideas to high order finite difference methods.     

The classical dispersion error analysis by Kreiss and Oliger, \cite{Kreiss1972}, predicts that high order methods are more efficient than low order methods when used for the simulation of wave propagation problems with smooth coefficients. High order finite difference methods are computationally efficient for solving hyperbolic PDEs and such methods are also easy to implement if the simulation domains are not too complex. 

In the distant past it was difficult to construct stable and high order accurate finite difference methods, but this challenge has now largely been overcome through the use of derivative approximations with the summation-by-parts (SBP) property \cite{Gustafsson2013}, and boundary and interface conditions enforced through the use of  ghost points \cite{Petersson2010,Sjogreen2012}, or by the simultaneous-approximation-term (SAT) method \cite{Carpenter1994}. 

Summation by parts operators for the second derivative $\frac{d^2}{dx^2}$ and their extension to variable coefficients can be found in \cite{Mattsson2012,Mattsson2004}. When these SBP operators are used together with the SAT method to impose  boundary or interface conditions for the wave equation in a second order in time and space form, the resulting discretization bears similarity with the symmetric interior penalty discontinuous Galerkin method (SIPDG) \cite{Grote2006}. As for SIPDG,  coercivity requires that such SBP-SAT methods use a mesh-dependent penalty term. This penalty parameter, which depends on the material properties and the SBP operator, must be large enough for the method to be stable \cite{Almquist2020,Appelo2007,Mattsson2008,Mattsson2009}. Precise bounds on this penalty parameter may not be easy to determine a priori \cite{Duru2014V,Wang2018}, especially in the presence of grid interface conditions \cite{Mattsson2008}, curvilinear grids \cite{Virta2014} and nonconforming grid interfaces \cite{Almquist2019,Wang2018,Wang2016}. 


In this paper, we present an energy-based discretization in the SBP-SAT framework. \reviii{Here, energy-based refers to the design principle advocated in \cite{Appelo2015,Appelo2018,Appelo2019}. That is, design a semi-discretization that is based on the sum of the kinetic and potential energy, leading to a new class of DG methods.} For problems with only Neumann boundary conditions, the proposed method is the same as the traditional SBP-SAT discretization in \cite{Mattsson2004}. \reviii{However, for Dirichlet boundary conditions and grid interface conditions, our method is different from the traditional SBP-SAT method in \cite{Appelo2007,Mattsson2009} and \cite{Mattsson2008}, respectively}, because our method does not use any mesh-dependent parameter. The method discretizes the wave equation in the velocity-displacement form, i.e. as a system with first order derivatives in time and second order derivatives in space. The resulting system of ordinary differential equations can then be evolved by a Runge-Kutta time integrator or a Taylor series method. As mentioned above, the method is inspired by our energy-based discontinuous Galerkin (DG) method  \cite{Appelo2015,Appelo2018,Appelo2019}.  We show that, just as the energy-based DG method, numerical dissipation can naturally be included in the method. We also present a general framework for deriving error estimates by normal mode analysis and perform the detailed analysis for a fourth order discretization for the Dirichlet problem. We will use the same approach as in \cite{Wang2017}  to prove fourth order convergence rate for the dissipative discretization and third order convergence rate for the energy-conserving discretization, which agree well with our numerical verification. 

The outline of the paper is as follows. We introduce the SBP concepts in Section 2. 
In Section 3, we construct our new SBP-SAT discretization for problems with boundaries and grid interfaces. Stability analysis and a priori error estimates by normal mode analysis are then derived. In Section 4, we present an efficient implementation of the method for higher dimensional problems. Numerical experiments verifying accuracy and robustness are presented in Section 5. We end with concluding remarks in Section 6.  

\section{Preliminaries}

Let $P$ and $Q$ be twice continuously differentiable functions in $[0, 1]$.  We define the standard inner product and norm in $L^2([0, 1])$ as 
\[
(P,Q) = \int_0^1 P Q dx,\quad \|P\|^2 = (P,P).
\]
\reviii{The integration-by-parts principle reads
\be\label{ibp}
\int_{0}^{1} P (bQ_x)_x dx = -\int_{0}^{1} P_x bQ_x dx + P bQ_x|_{0}^{1},
\eb
where $b$ is a continuously differentiable function in $[0,1]$. }
\reviii{It can be used to derive energy estimates for the wave equation
\be\label{wave1d}
U_{tt} = (c^2(x)U_x)_x,\quad x\in [0, 1],
\eb
with the energy
\be\label{energy}
E=\|U_t\|^2 + \|cU_x\|^2,
\eb
consisting of a kinetic and potential part.} 
To make matters concrete consider \reviii{\eqref{wave1d}} with a Dirichlet boundary condition at $x=0$ and a Neumann boundary condition at $x=1$,
\be\label{bc}
U(0,t) = f(t),\quad U_x(1,t)=g(t),
\eb
supplemented with compatible and smooth initial conditions for $U$ and $U_t$. Also assume that the wave speed  $c(x) \equiv \sqrt{b(x)}$ is smooth and positive. Then, multiplying \eqref{wave1d} by $U_t$ and integrating in space, the integration-by-parts principle \eqref{ibp} leads to
\be\label{dEdt}
\f{d}{dt} E = 2bU_t U_x|_0^1.
\eb

With homogeneous boundary conditions, the boundary contribution in \eqref{dEdt} vanishes and the continuous energy estimate states that the energy of \eqref{wave1d} is constant in time. 


\subsection{Summation-by-parts finite difference operators}
Consider the one dimensional domain $[0,1]$ discretized by an equidistant grid $\bs{x} = [x_1,x_2,\cdots,x_n]^T$ with grid spacing $h=1/(n-1)$.  The SBP finite difference operator for the approximation of \reviii{the} second derivative on the grid $\bs{x}$ is defined as follows.

%


\begin{definition}\label{defD}
A difference operator $D^{(b)}\approx \f{d}{dx}\left(b(x)\f{d}{dx}\right)$ with $b(x)>0$ is an SBP operator if $D$ can be decomposed as 
\be\label{Ddecompose}
D^{(b)} = H^{-1} (-A^{(b)}-b_1\bs{e_1d_1^T} + b_n \bs{e_n d_n^T} ),
\eb
where $H$ is diagonal and positive definite, $A^{(b)}$ is symmetric and positive semidefinite, $\bs{e_1}=[1,0,\cdots,0,0]^T$, $\bs{e_n}=[0,0,\cdots,0,1]^T$. The column vectors $\bs{d_1}$ and $\bs{d_n}$ contain coefficients for the approximation of the first derivative at the boundaries. The coefficients $b_1$ and $b_n$ are the function $b(x)$ evaluated at the boundaries.
\end{definition}

Let $\bs{p}$ and $\bs{q}$ be grid functions on $\bs{x}$. The operators $H$ and $A^{(b)}$ define a weighted discrete $L^2$ norm $\|\bs{p}\|^2_H = \bs{p}^TH\bs{p}$ and seminorm $\|\bs{p}\|^2_{A^{(b)}}=\bs{p}^TA^{(b)}\bs{p}$. The SBP property \eqref{Ddecompose} is a discrete analogue of the integration-by-parts principle,
\begin{equation}\label{sbp}
\bs{p^T}H(D^{(b)} \bs{q})=-\bs{p^T}A^{(b)} \bs{q} + \bs{p^T} (-b_1\bs{e_1 d_1^T}+b_n\bs{e_n d_n^T}) \bs{q}.
\end{equation}
If $\bs{p}$, $\bs{q}$ are $P(x)$, $Q(x)$ evaluated on the grid, then the term on the left-hand side of \eqref{sbp} is an approximation of the term on the left-hand side of \eqref{ibp} because $H$ is a quadrature \cite{Hicken2013}. 
On the right-hand side of \eqref{Ddecompose}, the term $-H^{-1}A^{(b)}$ is a discrete approximation of the Laplacian operator with homogeneous Neumann boundary conditions. In \eqref{sbp}, the term $\bs{p^T}A^{(b)} \bs{q}\approx \int_{0}^{1} bP_x Q_x dx$. \reviii{Therefore, the vector of constants, $[a,a,\cdots,a]\in\mathbf{R}^n$ for any $a\in\mathbf{R}$, is in the null space of $A^{(b)}$.} We make the following assumption of the operator $A^{(b)}$. 
\begin{assumption}\label{assumptionA}
In Definition \ref{defD}, the rank of $A^{(b)}$ is $n-1$ with any vector of constants in its null space. 
\end{assumption}

For constant coefficient (the superscript $b$ is dropped), the matrix \reviii{$hA$} corresponding to the second order accurate SBP operator takes the form
\[
\begin{bmatrix}
1 & -1  &&&& \\
-1 & 2 & -1 &&& \\
& -1 & 2 & -1 && \\
& & \ddots & \ddots & \ddots & \\
 &  & &-1&2&-1 \\
 &  & &&-1&1
\end{bmatrix},
\]
with eigenvalues 
\[
\lambda_j = {4}\sin^2\left(\frac{\pi(j-1)}{2n}\right), \ j=1,2,\cdots,n. 
\]
We observe that all eigenvalues are distinct. Since $\lambda_1 = 0$, the rank of $A$  is $n-1$, and Assumption \ref{assumptionA} is true. \reviii{For the second and fourth order SBP operators with constant coefficient, Assumption \ref{assumptionA} is proved in \cite{Eriksson2021} using the result from \cite{Eriksson2020}, and the explicit formulas for the Moore-Penrose inverse of $A$ are also derived. }

\subsection{Accuracy of SBP operators}\label{sec_accuracy_sbp}
Standard central finite difference stencils are used in the interior of the computational domain. To satisfy the SBP property, special non-centered difference stencils are used on a few grid points close to boundaries. In the interior where the central stencils are used, the truncation error is of even order, often denoted by $2p$ with $p=1,2,3,\cdots$. On a few grid points with the non-centered boundary stencils, the truncation error can at best be of order $p$ when $H$ is diagonal. We denote the accuracy of such SBP operators $(2p,p)$. Note that it is also common to refer to the accuracy of the operator and scheme as $2p^{th}$ order accurate. It is then important to be specific with the precise truncation error and convergence rate of the discretization. 

Though the boundary truncation error is order $p$, the convergence rate of the underlying numerical scheme can be higher in certain cases. This is in part due to the fact that the number of grid points with the less accurate boundary stencils is independent of grid spacing. The precise order of convergence rate depends on the equation, the spatial discretization and the numerical boundary conditions, see further the detailed error analysis in \cite{Wang2017,Wang2018b}. Below, in Section \ref{sec_nm}, we derive error estimates for the proposed scheme and we see that the choice of the SAT affects the convergence rate.

\section{An energy-based SBP-SAT finite difference method}

In this section, we derive the energy-based SBP-SAT discretization of the wave equation \eqref{wave1d}. First, we consider boundary conditions in Section \ref{sec_bc}. We show \reviii{that} our method is equivalent to that in \cite{Mattsson2009} for  Neumann boundary conditions, but is different from the discretization in \cite{Mattsson2009} for  Dirichlet boundary conditions.  We then derive error estimates in Section \ref{sec_nm} and consider grid interface conditions in Section \ref{sec_gic}.

\subsection{The boundary conditions}\label{sec_bc}
Our SBP-SAT discretization is based on the approximation of the unknown variable $U$ and its time derivative $U_t$. Therefore, we rewrite equation \eqref{wave1d} to a system with the first order derivative in time
\be\label{wave1dsys}
\begin{split}
U_t &= V,\\
V_t &= (b(x)U_x)_x.
\end{split}
\eb
The energy-based SBP-SAT finite difference approximation of \eqref{wave1dsys} with the boundary condition \eqref{bc} is
\begin{align}
A^{(b)}(\bs{u}_t-\bs{v}) & =b_1\bs{d_1} (\bs{e}_{\bs{1}}^T \bs{v}-f_t) + \alpha \bs{d_n } (\bs{d}_{\bs{n}}^T \bs{u}-g),\label{wave1dsingle1}\\
\bs{v}_t &= D^{(b)} \bs{u}-b_n H^{-1}\bs{e_n} (\bs{d}_{\bs{n}}^T \bs{u}-g)+\beta H^{-1}\bs{e_1} (\bs{e}_{\bs{1}}^T \bs{v}-f_t),\label{wave1dsingle2}
\end{align}
where $\bs{u}$ and $\bs{v}$ are grid functions that approximate $U$ and $V$, respectively. On the right-hand side of \eqref{wave1dsingle1}, the first term imposes weakly the time derivative of the Dirichlet boundary condition $U_t(0, t)=f_t(t)$ \reviii{(note that the constant level of the solution is uniquely determined by the initial data and, being a constant, is not affected by this boundary condition)}. This penalty term \reviii{affects the stencils on a few grid points near the left boundary because of the weights in $\bs{d_1}$.} The second term is a dissipative term controlled by $\alpha$ and  contributes to a few grid points near the right boundary. On the left-hand side, $\bs{u}_t$ is approximately equal to $\bs{v}$. The symmetric positive semidefinite matrix $A^{(b)}$ is multiplied by $\bs{u}_t-\bs{v}$. In the traditional way of imposing the Dirichlet boundary condition in the SBP-SAT finite difference method, the corresponding penalty term is based on $U(0, t)=f(t)$ and involves a mesh-dependent penalty parameter \cite{Mattsson2009}. Such mesh dependent parameters are not needed in our energy-based discretization. 

\reviii{
The last term on the right-hand side of \eqref{wave1dsingle2} is a dissipation term that has contribution only on the first grid point and is controlled by the parameter $\beta$. We note the Neumann boundary condition is imposed weakly in the same way as in \cite{Mattsson2004}.  To see this, consider  \eqref{wave1dsingle1}-\eqref{wave1dsingle2} without terms for the Dirichlet boudary condition and the dissipation. Then,  \eqref{wave1dsingle1} is reduced to $\bs u_{tt}=\bs v_t$. Replacing $\bs v_t$ by $\bs u_{tt}$ in \eqref{wave1dsingle2} gives an equivalent formulation as in \cite{Mattsson2004}.
}

 The stability property of the semi-discretization is stated in the following theorem. 

\begin{theorem}\label{thm_bc}
With homogeneous boundary conditions, the energy-based SBP-SAT discretization \eqref{wave1dsingle1}-\eqref{wave1dsingle2} with $\alpha\leq 0$ and $\beta\leq 0$  satisfies 
\[
\f{d}{dt} E_H = 2\alpha (\bs{d}_{\bs{n}}^T \bs{u})^2  + 2\beta (\bs{e}_{\bs{1}}^T \bs{v})^2\leq 0.
\]
The discrete energy $E_H$ is defined as 
\[
 E_H \equiv (\|u\|_{A^{(b)}}^2+\|v||_H^2), 
\]
and is the discrete analogue of the continuous energy \eqref{energy}.
\end{theorem}

\begin{proof}
Consider homogeneous boundary conditions with $f=g=0$ in \eqref{bc}. \reviii{We multiply from the left of \eqref{wave1dsingle1} by $\bs{u}^T$, and \eqref{wave1dsingle2} by $\bs{v}^T H$, and  obtain}
\begin{align}
\bs{u}^T A^{(b)} (\bs{u}_t-\bs{v}) & = b_1\bs{u}^T \bs{d_1} \bs{e}_{\bs{1}}^T \bs{v} + \alpha \bs{u}^T \bs{d_n} \bs{d}_{\bs{n}}^T \bs{u} ,\label{wave1dsingle3}\\
\bs{v}^T H \bs{v}_t &= \bs{v}^T H D^{(b)} \bs{u}-b_n \bs{v}^T \bs{e_n} \bs{d}_{\bs{n}}^T \bs{u}+\beta \bs{v}^T \bs{e_1} \bs{e}_{\bs{1}}^T \bs{v}.\label{wave1dsingle4}
\end{align}
Adding \eqref{wave1dsingle3} and \eqref{wave1dsingle4}, we have 
\[
\f{d}{dt} E_H \equiv \f{d}{dt} (\|u\|_{A^{(b)}}^2+\|v||_H^2) = 2\alpha (\bs{d}_{\bs{n}}^T \bs{u})^2  +2 \beta (\bs{e}_{\bs{1}}^T \bs{v})^2\leq 0,
\]
if $\alpha\leq 0$ and $\beta\leq 0$. 
\end{proof}
If $\alpha< 0$ and $\beta< 0$, the discrete energy $E_H$ is dissipated even though the continuous energy is conserved. With $\alpha=\beta=0$, the discrete energy is constant in time. \revii{All four penalty terms in \eqref{wave1dsingle1}-\eqref{wave1dsingle2} have no  mesh-dependent parameters.}


\revii{In the first semi-discretized equation \eqref{wave1dsingle1}, the time derivative of the unknown variable $\bs{u}$ is given implicitly, since  $\bs{u}_t$ is multiplied by $A^{(b)}$.}  The matrix $A^{(b)}$ is banded and symmetric positive semi-definite, with nullspace consisting of vectors of constants. Since  both $\bs{d_1}$ and $\bs{d_n}$ are consistent finite difference stencils for the first derivative, the right hand side will always be in the range of $A^{(b)}$. In other words,  the solution exists but is only unique up to a constant. A unique solution can be obtained with an additional constraint. Here, we require that the sum of all elements in  $\bs{u}_t-\bs{v}$ are zero, consistent with the equation  $U_t=V$. Numerically, this constraint can be taken into account by the Lagrange multiplier technique. With the new variables
\[\tilde{A}^{(b)}=\begin{bmatrix}
 &  & & 1 \\
 & A^{(b)} & & \vdots \\
 & & &  1 \\
  1 & \cdots & 1& 0
\end{bmatrix}
,\ \tilde{\bs{u}}= \begin{bmatrix}
\\
\bs u \\
\\
\mu
\end{bmatrix}
,\ \tilde{\bs{v}}= \begin{bmatrix}
\\
\bs v \\
\\
\nu
\end{bmatrix}
,\ \tilde{\bs{d_1}}= \begin{bmatrix}
\\
\bs{d_1} \\
\\
0
\end{bmatrix}
,\ \tilde{\bs{d_n}}= \begin{bmatrix}
\\
\bs{d_n} \\
\\
0
\end{bmatrix},
\]
equation \eqref{wave1dsingle1} is replaced by
\begin{equation}\label{wave1dLM}
\tilde{A}^{(b)}(\tilde{\bs{u}}_t-\tilde{\bs{v}})  =b_1\tilde{\bs{d_1}} (\bs{e}_{\bs{1}}^T \bs{v}-f_t) + \alpha \tilde{\bs{d_n}} (\bs{d}_{\bs{n}}^T \bs{u}-g),
\end{equation}
which is nonsingular. The auxiliary variables $\mu$ and $\nu$ satisfy $\mu_t\approx\nu$. \revii{Alternatively, we can use the pseudoinverse of $A^{(b)}$. Since the right-hand side of \eqref{wave1dsingle1} is a summation of rank-one vectors, we only need a few columns of the pseudoinverse of $A^{(b)}$, which can be computed by using the analytical formula in \cite{Eriksson2020,Eriksson2021} for constant coefficient problems and $p=2$ or 4. After that, the resulting system of first order ordinary differential equations can be advanced explicitly in time by using standard time integrators, for example Runge-Kutta methods. }

\begin{remark} \label{rem:1}
At first glance, it appears that the formulation \eqref{wave1dsingle1} would have a higher computational complexity than comparable methods but, as we show in Section \ref{sec_high}, for constant coefficient systems there is a fast direct algorithm that results in a linear (in the number of degrees of freedom) complexity. For variable coefficients, we will illustrate by numerical examples that the preconditioned conjugate gradient method only requires a very small number of iterations per timestep to converge. 
\end{remark}


\subsection{Error estimates}\label{sec_nm}
As discussed in Section \ref{sec_accuracy_sbp}, the $2p^{th}$ order accurate SBP operators with diagonal norms are only $p^{th}$ order accurate on a few grid points near boundaries. In this section, we derive  error estimates and analyze the effect of the $p^{th}$ order truncation error on the overall convergence rate of the discretization. We note that the energy-based discretization for the Neumann problem is the same as the traditional SBP-SAT method \cite{Mattsson2004} and for this problem the error estimates derived in  \cite{Wang2017} already applies.  Below, we consider the problem with Dirichlet boundary conditions by first outlining the general approach \cite{Gustafsson2013} for error analysis and then specializing to the case when $p=2$. \reviii{As will be seen, dissipation at the Dirichlet boundary conditions affects the overall convergence rate.}  We note that the influence of dissipation for a discretization of the wave equation is also considered in \cite{Svard2019}. 

Consider the following half line problem 
\begin{align*}
U_t &= V, \\
V_t &=U_{xx},
\end{align*}
in the domain $x\in [0,\infty)$ with the Dirichlet boundary condition  $U(0,t)=f(t)$ and $t\in [0,t_f]$ for some final time $t_f$. The corresponding energy-based discretization is 
\begin{align}
A(\bs{u}_t-\bs{v})&=\bs{d_1}(\bs{e}_{\bs{1}}^T \bs{v}-f_t),\label{eqn1Dirichlet}\\
\bs{v}_t &= D \bs{u} + \beta H^{-1} \bs{e}_{\bs{1}} (\bs{e}_{\bs{1}}^T \bs{v}-f_t).\label{eqn2Dirichlet}
\end{align}
When $\beta\leq 0$, the discretization satisfies an energy estimate as in Theorem \ref{thm_bc}. We will see below that the energy-dissipative discretization with $\beta<0$ gives a higher convergence rate than the energy-conserving discretization with $\beta=0$.

Let $\bs{\xi}=[\xi_{1}, \xi_{2}, \cdots]^T$ and $\bs{\zeta}=[\zeta_{1}, \zeta_{2}, \cdots]^T$ be the pointwise error vector with $\xi_{j}=u_j(t)-U(x_j,t)$ and $\zeta_{j}=v_j(t)-V(x_j,t)$.  We then have the error equations
\begin{align*}
A(\bs{\xi}_t-\bs{\zeta})&=\bs{d_1}\bs{e}_{\bs{1}}^T \bs{\zeta},\\
\bs{\zeta}_t &= D \bs{\xi} + \beta H^{-1} \bs{e}_{\bs{1}} \bs{e}_{\bs{1}}^T \bs{\zeta}+\bs{T},
\end{align*}
where $\bs{T}$ is the truncation error vector. \reviii{Note that there is no truncation error in the first equation, because the equation is satisfied exactly by the true solution on the grid.} We partition the truncation error into two parts, the boundary truncation error $\bs{T^B}$ and the interior truncation error $\bs{T^I}$ such that $\bs{T}=\bs{T^B}+\bs{T^I}$.  The only nonzero elements of $\bs{T^B}$ are the first $r$ elements  corresponding to the boundary stencil of $D$ and are of order $\mathcal{O}(h^p)$, where $r$ depends on $p$ but not $h$. In $\bs{T^I}$, the first $r$ elements are zero and the rest are of order $\mathcal{O}(h^{2p})$ corresponding to the interior stencil of $D$. 

We partition the error as $\bs{\xi}=\bs{\xi^I}+\bs{\xi^B}$ and $\bs{\zeta}=\bs{\zeta^I}+\bs{\zeta^B}$. The first terms $\bs{\xi^I},\bs{\zeta^I}\sim \mathcal{O}(h^{2p})$ are caused by the interior truncation error $\bs{T^I}$ and can be estimated by the energy technique for SBP methods. It is often the second part, $\bs{\xi^B},\bs{\zeta^B}$ caused by the boundary truncation error $\bs{T^B}$, that determine the overall convergence rate of the scheme. 
We note that $\bs{\xi^B},\bs{\zeta^B}$ satisfy the error equations
\begin{align}
A((\bs{\xi^B})_t- \bs{\zeta^B})&=\bs{d_1}\bs{e}_{\bs{1}}^T \bs{\zeta^B},\label{ve1}\\
(\bs{\zeta^B})_t &= D \bs{\xi^B} + \beta H^{-1} \bs{e}_{\bs{1}} \bs{e}_{\bs{1}}^T \bs{\zeta^B}+\bs{T^B}.\label{ve2}
\end{align}
For convenience, we define the $h$-independent quantities 
\[
\wb{A} = hA,\quad \wb{H} = \f{1}{h}H,\quad \wb{\bs{d_1}} = h \bs{d_1},
\]
and the new variables  
\begin{equation}\label{newvariables}
\bs\delta_t=\bs{\zeta^B} \text{ and } \bs\ve = \bs{\xi^B}-\bs\delta.
\end{equation}
\reviii{We have the relation $\bs\ve_t = (\bs{\xi^B})_t- \bs{\zeta^B}$.}

Next, we take the Laplace transform of the error equations \eqref{ve1}-\eqref{ve2} in time. \reviii{With exact initial data}, we obtain the difference equations
\begin{align}
\wb{A} \wh{\bs{\ve}}-\wb{\bs{d_1}} \bs{e}_{\bs{1}}^T \wh{\bs{\delta}}&=\bs{0},\label{Tve1}\\
\bs{e}_{\bs{1}}\wb{\bs{d_1}}^T  \wh{\bs{\ve}} + (\wt{s}^2\wb{H}+\wb{A}+ \bs{e}_{\bs{1}} \wb{\bs{d_1}}^T  + \wb{\bs{d_1}} \bs{e}_{\bs{1}}^T - \beta\wt{s} \bs{e}_{\bs{1}} \bs{e}_{\bs{1}}^T ) \wh{\bs{\delta}}  &=  h^2\wb{H}\wh{\bs{T^B}},\label{Tve2}
\end{align}
where $\wh{\bs{\ve}}$ and $\wh{\bs{\delta}}$ are the Laplace-transform of ${\bs{\ve}}$ and ${\bs{\delta}}$, respectively. We also use the notation $\wt{s}=sh$, where $s$ is the dual of time. \reviii{Note that when deriving \eqref{Tve2}, we have used \eqref{Tve1} and the identity $D=H^{-1}(-A-\bs{e_1}\bs{d}_{\bs{1}}^T)$ because of the half-line problem. }

In Laplace space, we solve the difference equations \eqref{Tve1}-\eqref{Tve2} and use \eqref{newvariables} to derive  an error estimate for $\wh{\bs{\xi^B}}$. The corresponding error estimate for  $\bs{\xi^B}$ in physical space can then be obtained by Parseval's relation. The precise estimate depends on the operators in \eqref{Tve1}-\eqref{Tve2}. Below we consider the SBP operator with accuracy $(2p,p)=(4,2)$ constructed in \cite{Mattsson2004}. \reviii{The accuracy analysis follows the same procedure when other SBP operators are used.}

\begin{theorem}
With the SBP operator of accuracy order (4,2) from \cite{Mattsson2004}, the method \eqref{eqn1Dirichlet}-\eqref{eqn2Dirichlet} has convergence rate four with a  dissipative discretization $\beta<0$. For the energy-conserving discretization with $\beta=0$, the convergence rate is three. 
\end{theorem}

\begin{proof}
In this case, $\wh{\bs{T^B}}$ in \eqref{Tve2} is to the leading order 
\[
\wh{\bs{T^B}} = h^2 \wh U_{xxxx}(0,s) \left[\frac{11}{12}, -\frac{1}{12}, \frac{5}{516}, \frac{11}{588}, 0, 0, \cdots\right]^T.
\]

In  \eqref{Tve1}, the matrix $\wb{A}$ has boundary stencils in the first four rows and  repeated interior stencil from row five. \reviii{The only nonzeros of $\wb{\bs{d_1}} \bs{e}_{\bs{1}}^T$ are in the first four rows.}  Therefore, from row five the difference equation takes the form
\[
\f{1}{12}  \wh{{\ve}}_{j-2} - \f{4}{3}\wh{{\ve}}_{j-1} +\f{5}{2} \wh{{\ve}}_{j} - \f{4}{3}\wh{{\ve}}_{j+1} +\f{1}{12}\wh{{\ve}}_{j+2} = 0,\quad j=5, 6,\cdots.
\]
The corresponding characteristic equation 
\[
\f{1}{12}  \lambda^4 - \f{4}{3}\lambda^3 +\f{5}{2} \lambda^2 - \f{4}{3}\lambda +\f{1}{12}= 0
\]
has four solutions $7-4\sqrt{3}\approx 0.0718$, $7+4\sqrt{3}\approx 13.9282$, and a double root 1. The only admissible solution satisfying $|\lambda|<1$ is $\lambda=7-4\sqrt{3}$. As a consequence, the elements of the vector $\wh{\bs{\ve}}$ can be written as 
\[
 \wh{\bs{\ve}} = [\wh{\ve}_1, \wh{\ve}_2, \wh{\ve}_3, \sigma, \sigma \lambda, \sigma \lambda^2, \sigma \lambda^3,\cdots]^T,
\]
with four unknowns $\wh{\ve}_1$, $\wh{\ve}_2$, $\wh{\ve}_3$ and $\sigma$.  \reviii{Note that it is also possible to use three unknowns $\wh{\ve}_1, \wh{\ve}_2, \sigma$. We formulate the linear system with four unknowns to match the number of equations. In this case, we have the relation $\wh{\ve}_3=\sigma\lambda^{-1}$.}
These four unknowns, together with the unknowns in  $\wh{\bs{\delta}}$, are involved in the first four equations of \eqref{Tve1}. To solve for them, we also need to consider \eqref{Tve2}.

The difference equation from row five of \eqref{Tve2} takes the form 
\[
\f{1}{12}  \kappa^4 - \f{4}{3}\kappa^3 +\left(\f{5}{2}+\wt{s}^2\right) \kappa^2 - \f{4}{3}\kappa +\f{1}{12}= 0,
\]
and has two admissible roots 
\[\kappa_{1} = 7-4\sqrt{3}+\mathcal{O}(\wt{s}^2), \quad \kappa_{2} = 1-\mathcal{O}(\wt{s}).
\]
We note that the second admissible root $\kappa_{2}$ is a slowly decaying component at $\wt s\rightarrow 0^+$. The vector $\wh{\bs{\delta}}$ can then be written as 
\[
 \wh{\bs{\delta}} = [\wh{\delta}_1, \wh{\delta}_2, \sigma_{1} + \sigma_{2}, \sigma_{1} \kappa_{1}+\sigma_{2} \kappa_{2}, \sigma_{1} \kappa_{1}^2+\sigma_{2} \kappa_{2}^2, \sigma_{1} \kappa_{1}^3+\sigma_{2} \kappa_{2}^3, \cdots]^T,
\]
with four unknowns $\wh{\delta}_1$, $\wh{\delta}_2$, $\sigma_{1}$ and $\sigma_{2}$.

The first four equations of \eqref{Tve1} and the first four equations of \eqref{Tve2} lead to the eight-by-eight boundary system 
\begin{equation}\label{bs}
C(\wt{s},\beta)\z =\hat T_{uv},
\end{equation}
where the unknown vector $\z$ and the right-hand side vector $T_{uv}$ are 
\begin{align}
\z &= [\wh{\ve}_1, \wh{\ve}_2, \wh{\ve}_3, \sigma,\wh{\delta}_1, \wh{\delta}_2, \sigma_{1}, \sigma_{2}]^T, \label{vec_z} \\
\hat T_{uv} &= \left[0,0,0,0,\f{187}{576},-\f{59}{576},\f{5}{576},\f{11}{576}\right]^Th^4 \wh U_{xxxx}(0,s). \notag
\end{align}
\reviii{The nonzeros of $\hat T_{uv}$ are the nonzeros of $\wh{\bs{T^B}}$ scaled by the first four diagonal elements of $\wb H$, i.e. $\frac{17}{48}, \frac{59}{48}, \frac{43}{48}, \frac{49}{48}$.}
From \eqref{newvariables}, we have 
\begin{align}
\wh{\bs{\xi^B}}&=\wh{\bs{\ve}}+\wh{\bs{\delta}} \label{veuv}\\
&=[\wh{\ve}_1+\wh{\delta}_1, \wh{\ve}_2+\wh{\delta}_2, \wh{\ve}_3+\sigma_{1}+\sigma_{2}, \sigma+\sigma_{1} \kappa_{1}+\sigma_{2} \kappa_{2},\sigma \lambda+\sigma_{1} \kappa_{1}^2+\sigma_{2} \kappa_{2}^2,\cdots]^T ,\notag
\end{align}
which depends on all the eight unknowns in the vector $\z$. 

\reviii{To analyze convergence rate, we shall consider the solution to the boundary system in the vicinity of $\wt s\rightarrow 0^+$ corresponding to the asymptotic region when $h\rightarrow 0$. When the scheme is  stable with $\beta\leq 0$, the boundary system is non-singular for all $Re(\wt s)>0$ \cite{Gustafsson2013}. However, the solution to the boundary system may depend on $h$, and the precise dependence is important to the convergence rate. To this end, we analyze the inverse of $C(\wt s,\beta)$, and the components of $\z$ and $\wh{\bs{\xi^B}}$ in the vicinity of $\wt s\rightarrow 0^+$. }


We start by considering the boundary system \eqref{bs} with $\wt s=0$. Here, the matrix $C(0):=C(0,\beta)$ is independent of $\beta$ and \reviii{takes the form}
\begin{equation*}
C(0)= \begin{bmatrix}
1.1250  & -1.2292 &   0.0833  &  0.0208  &  1.8333   &      0 &        0 &        0 \\
   -1.2292  &  2.4583  & -1.2292    &     0 &  -3.0000   &      0 &        0&         0 \\
    0.0833  & -1.2292   & 2.2917   &-1.2232  &  1.5000   &      0  &       0  &       0\\
    0.0208    &     0 &  -1.2292 &   2.3630 &  -0.3333  &       0  &       0  &       0\\
   -1.8333 &   3.0000  & -1.5000  &  0.3333 &  -2.5417  &  1.7708 &  -1.3912 &  -1.0625\\
         0  &       0  &       0   &      0  &  1.7708  &  2.4583 &  -1.2292 &  -1.2292\\
         0  &       0   &      0   &      0  & -1.4167  & -1.2292  &  2.2038  &  1.1458\\
         0  &       0   &      0   &      0  &  0.3542   &      0  & -1.0595 &  -0.0208
\end{bmatrix}.
\end{equation*}
It is singular with one eigenvalue equal to zero, i.e. the so-called determinant condition is not satisfied. Since the matrix $C(\wt s, \beta)$ cannot be inverted at $\wt s=0$, we take a similar approach as in \cite{Nissen2012,Wang2017} and consider $Re(\wt s)=\eta h$ for a small constant $\eta>0$ independent of $h$. \revii{We also refer to \cite{Gustafsson2013} for this technique. }

The Taylor series of $C(\wt s, \beta)$ at $\wt s=0$ can be written as 
\be\label{ts}
C(\wt s,\beta)=C(0)+\wt s C'(0,\beta)+\f{{\wt s}^2}{2}C''(0,\beta)+\mathcal{O}({\wt s}^3),
\eb
where  $C'(0,\beta)$ and $C''(0,\beta)$ are the first and second derivative of $C(\wt s,\beta)$ with respect to $\wt s$ at $\wt s=0$, respectively. We perform a singular value decomposition of the singular matrix $C(0)=M\Sigma N^*$ with two unitary matrices $M$ and $N$. Substituting into the Taylor series, we obtain the boundary system to the leading order
\be\label{bs1}
(\Sigma + \wt s C_\beta)\bz = M^* \hat T_{uv},
\eb
where $\bz = N^* \z$ and $C_\beta=M^* C'(0,\beta) N$. The solution can be written as 
\[
\bz = (\Sigma + \wt s C_\beta)^{-1} M^* \hat T_{uv},
\]
and consequently we have 
\[
\z=N  (\Sigma + \wt s C_\beta)^{-1} M^* \hat T_{uv}.
\]
We note that $\Sigma$ is a diagonal matrix where the first seven diagonal components are nonzero and the last diagonal component equals to zero. The last diagonal element of $C_\beta$, denoted by $(C_\beta)_{88}$, is crucial. We find that $(C_\beta)_{88}$ is zero when $\beta=0$ and nonzero when $\beta<0$. We analyze these two cases separately.

When $\beta<0$, we have $(C_\beta)_{88}\neq 0$. In this case, we have 
\[
(\Sigma + \wt s C_\beta)_{ij}=\begin{cases}
\mathcal{O}(1), \ 1\leq i=j\leq 7, \\
\mathcal{O}(\wt s), \ i=j=8, \\
\mathcal{O}(\wt s) \text{ or } 0, \ i\neq j.
\end{cases}
\]
\reviii{We use Gaussian elimination to reduce the linear system \eqref{bs1} to a triangular form.} The resulting upper triangular matrix has diagonal elements $\mathcal{O}(1)$, except the last diagonal element $\mathcal{O}(\wt s)$. Using that all elements of $M^* \hat T_{uv}$ are $\mathcal{O}(h^4)$, the backward substitution procedure gives the solution $\bz$ in the form such that its first seven elements are $\mathcal{O}(h^4)$ and the last element is $\mathcal{O}({\wt s}^{-1} h^4)\sim \mathcal{O}(h^3)$. 

The dominating error component $\mathcal{O}(h^3)$ in $\bz$ is a potential source of accuracy reduction. To analyze its effect to the convergence rate, we only consider this dominating component in $\bz$, that is  $[0,0,0,0,0,0,0,1]^T Kh^3$ for some constant $K$. By the relation $\z=N\bz$, the corresponding part of $\z$ can be computed as $\z=N [0,0,0,0,0,0,0,1]^T Kh^3$. \reviii{A direct calculation of $\z$ shows that its components satisfy the following relations,}
\be\label{zform}
\wh{\ve}_1+\wh{\delta}_1= \wh{\ve}_2+\wh{\delta}_2= \wh{\ve}_3+\sigma_{1} =0,\ \sigma+\sigma_{1}\kappa_{1}=0,\ \sigma_{2}=0.
\eb     

The last relation $\sigma_{2}=0$ is important because in the error vector \eqref{veuv}, the variable $\sigma_{2}$ is multiplied with the slowly decaying component $\kappa_{2}$, which is now eliminated. By using the first relation in \eqref{zform} and the relation $\lambda=\kappa_{1}+\mathcal{O}({\wt s}^2)$, the error vector \eqref{veuv} becomes
\begin{align*}
\wh{\bs{\xi^B}}=h^3[0,0,0,0,\mathcal{O}({\wt s}^2), \mathcal{O}({\wt s}^4),\cdots]^T. 
\end{align*}
Therefore, the dominating error component $\mathcal{O}(h^3)$ in $\bz$ does not lead to a convergence rate reduction. As a consequence, the error $\wh{\bs{\xi^B}}$ is determined by the first seven elements of $\bz\sim\mathcal{O}(h^4)$. It is then straightforward to compute
\begin{align*}
\left\|\wh{\bs{\xi^B}}\right\|_h\leq \wt K h^{4.5} |\wh U_{xxxx}(0,t)|
\end{align*}
for some constant $\wt K$. 
By using Parseval's relation, we have 
\[
\int_0^{t_f} \|\bs{\xi^B}\|_h dt \leq K h^{4.5}e^{2\eta t_f} \int_0^{tf} |U_{xxxx}(0,t)|dt
\]
for some constant $K$. In the above, we have used the argument ``future cannot affect past''  \cite[pp.~294]{Gustafsson2013}. Finally, the overall error $\bs{\xi}$ is in fact determined by the interior scheme. We conclude that the scheme has a fourth order convergence rate when $\beta<0$.

Now, we consider the case $\beta=0$. Since $(U_c^*C'(0,\beta)V_c)_{88}=0$, it is necessary to include the quadratic term of the Taylor series \eqref{ts} in the boundary system analysis. A direct calculation gives  $(U_c^*C''(0,\beta)V_c)_{88}\neq 0$. The solution $\bz$ to the boundary system is then in the form such that its first seven components are $\mathcal{O}(h^3)$ and the last component is $\mathcal{O}(h^2)$. The first seven components $\mathcal{O}(h^3)$ lead to $\|\wh{\bs{\xi^B}}\|_h  \leq K_3 h^{3} |\wh U_{xxxx}(0,t)|$ for some constant $K_3$, and the \revii{dominating} error component $\mathcal{O}(h^2)$ does not lead to further reduction in convergence rate for the same reason as the case with $\beta<0$.  \reviii{Hence, the convergence rate is three when $\beta=0$. Note that in this case, the slowly decaying component $\kappa_{2}$ does not vanish in $\wh{\bs{\xi^B}}$.}  This concludes the proof.
\end{proof}

 
\subsection{Interface conditions}\label{sec_gic}
In heterogeneous materials, a multi-block finite difference discretization can be advantageous. Different grid spacings can be used in different blocks to adapt to the velocity structure of the material. If the material property is discontinuous, the material interfaces can be aligned with block boundaries so that high order accurate discretization can be constructed in each block.

As a model problem, we again consider the wave equation \eqref{wave1dsys} in the domain [-1,1]. The parameter $b(x)$ is smooth in $(-1,0)$ and $(0,1)$, but discontinuous at $x=0$. 
In the stability analysis, we omit terms corresponding to the boundaries $x=0,1$ and focus on the interface contribution. The energy \eqref{energy} is conserved in time if we impose the interface conditions 
\begin{align}
\lim_{\epsilon\rightarrow 0^+} U(-\epsilon,t) &= \lim_{\epsilon\rightarrow 0^+}  U(\epsilon,t), \label{int1}\\
\lim_{\epsilon\rightarrow 0^+} b(-\epsilon)U_x(-\epsilon,t) &= \lim_{\epsilon\rightarrow 0^+}  b(\epsilon) U_x(\epsilon,t). \label{int2}
\end{align}
In this case,  at $x=0$ the solution is continuous but its derivative is discontinuous. 

We proceed by deriving an energy-based SBP-SAT discretization. To distinguish notations in the two subdomains, we use a tilde symbol on top of the variables and operators in $[0, 1]$. The semi-discretization reads 

\begin{align}
&A^{(b)} (\bs{u}_t-\bs{v})  =-\tau b_n \bs{d_n} (\bs{e}_{\bs{n}}^T\bs{v}-\bs{\tilde e}_{\bs{1}}^T \bs{\tilde v}),\label{semi_u} \\
&\bs{v}_t = D^{(c)} \bs{u}-(1-\tau) H^{-1} \bs{e_n} (b_n \bs{d}_{\bs n}^T \bs{u}-\tilde b_1\bs{\tilde d}_{\bs 1}^T {\tilde u} )+\gamma H^{-1} \bs{e_n} (\bs{e}_{\bs n}^T \bs{v}-\bs{\tilde e}_{\bs{1}}^T \bs{\tilde v}), \label{semi_v} 
\end{align}
in the subdomain $[-1,0]$, and
\begin{align}
&\tilde A^{(b)} (\bs{\tilde u}_t-\bs{\tilde v})  =(1-\tau) \tilde b_1 \bs{\tilde d_1} (\bs{\tilde e}_{\bs 1}^T \bs{\tilde v}-\bs{e}_{\bs n}^T \bs{v}), \label{semi_u_tilde}  \\
&\bs{\tilde v}_t = \tilde D^{(b)} \bs{\tilde u}+\tau \tilde H^{-1} \bs{\tilde e_1} (\tilde b_1 \bs{\tilde d}_{\bs 1}^T \bs{\tilde u}-b_n \bs{d}_{\bs n}^T \bs{u} )+\gamma \tilde H^{-1} \bs{\tilde e_1} (\bs{\tilde e}_{\bs 1}^T \bs{\tilde v}-\bs{e}_{\bs n}^T \bs{v}). \label{semi_v_tilde}
\end{align}
in the subdomain $[0,1]$. Similar to the Dirichlet boundary condition, we impose continuity of the time derivative of the solution, instead of continuity of the solution itself. Unlike in \cite{Virta2014}, no mesh-dependent parameter is needed to impose interface conditions by the SAT method. The stability property is summarized in the following theorem.

\begin{theorem}
The semi-discretization \eqref{semi_u}-\eqref{semi_v_tilde} satisfies
\[
\f{d}{dt}E_H = 2\gamma (\bs{e}_{\bs{n}}^T \bs{v}-\bs{\tilde e}_{\bs 1}^T \bs{\tilde v})^2\leq 0,
\]
where the discrete energy is defined as $E_H \equiv \|\bs{u}\|_{A^{(b)}}^2 + \|\bs{v}\|_H^2 + \|\bs{\tilde u}\|_{\tilde A^{(b)}}^2 + \|\bs{\tilde v}\|_{\tilde H}^2$ 
 for any $\tau$ and $\gamma\leq 0$. 
\end{theorem}

\begin{proof}
We multiply \eqref{semi_u} by $\bs{u}^T$, \eqref{semi_v} by \reviii{$\bs{v}^TH$}, \eqref{semi_u_tilde} by $\bs{\tilde u}^T$, \eqref{semi_v_tilde} by \reviii{$\tilde{\bs{v}}^T\tilde H$}. After adding all the four equations, we obtain 
\[
\f{d}{dt}E_H = 2\gamma (\bs{e}_{\bs{n}}^T \bs{v}-\bs{\tilde e}_{\bs 1}^T \bs{\tilde v})^2,
\]
where the discrete energy is $E_H = \|\bs{u}\|_{A^{(b)}}^2 + \|\bs{v}\|_H^2 + \|\bs{\tilde u}\|_{\tilde A^{(b)}}^2 + \|\bs{\tilde v}\|_{\tilde H}^2$. \qed
\end{proof}
\revii{The penalty parameters $\tau$ and $\gamma$ do not depend on the mesh size.} We have $\f{d}{dt}E_H\leq 0$ when $\gamma\leq 0$. In particular, if $\gamma=0$, then the discrete energy is conserved in time. 
We note that the linear system involving $A^{(b)}$ and $\tilde A^{(b)}$ can be solved separately in each domain, thus resulting a linear computational complexity with respect to the number of degrees of freedom. 

\section{Discretization in higher space dimensions}\label{sec_high}
The one-dimensional discretization technique can be generalized to higher dimensional problems by tensor products. 
As an example, we consider the wave equation in two space dimensions (2D)
 \be\label{wave2dsys}
\begin{split}
U_t &= V,\\
V_t &= (a(x,y)U_x)_x + (b(x,y)U_y)_y +F,
\end{split}
\eb
in the domain $\Omega=[0,1]^2$ with Dirichlet boundary conditions 
\[
U(x,y,t) = f(x,y,t) \text{ on } \p\Omega,
\] 
and a forcing function $F$. 
We discretize $\Omega$ by a Cartesian grid with $n_x$ points in $x$ and $n_y$ points in $y$. The semi-discretization can be written as
\begin{align}
\mathbf{A} (\bs{u}_t-\bs{v})  = &\ {d_W} H_y \left( e_W^T \bs{v}- \bs{f_{tW}}\right)  
 - d_E H_y(e_E^T \bs{v}- \bs{f_{tE}})   \notag\\
 &+ d_S H_x ( e_S^T \bs{v}- \bs{f_{tS}})  
 - d_N H_x ( e_N^T \bs{v}- \bs{f_{tN}}),  \label{semi2D1} \\
\bs{v_t} = &\ \mathbf{D}\bs{u}+\bs{F}+\theta \mathbf{H}^{-1} [e_W H_y (e_W^T\bs{v}-\bs{f_{tW}})+e_E H_y (e_E^T\bs{v}-\bs{f_{tE}}) \notag\\
&+e_S H_x (e_S^T\bs{v}-\bs{f_{tS}})+e_N H_x (e_N^T\bs{v}-\bs{f_{tN}})]. \label{semi2D2}
\end{align}
The operator $\mathbf{D}$ in \eqref{semi2D2} approximates the second derivative with variable coefficients in 2D and is defined as 
\begin{align*}
\mathbf{D} & = \sum_{i=1}^{ny}  D^{a_{:,i}}\otimes E_y^i + \sum_{j=1}^{nx}  E_x^j\otimes D^{b_{j,:}},
\end{align*} 
where $D^{a_{:,i}}$ approximates $\frac{\partial}{\partial x}\left(a(x,y_i)\frac{\partial}{\partial x}\right)$ and the only nonzero element in the $n_y$ by $n_y$ matrix $E_y^i$ is $E_y^i(i,i) = 1$. The operators corresponding to the term in the $y$-direction is defined similarly. The operator $\mathbf{H}=H_x\otimes H_y$ defines the 2D SBP norm and quadrature. In addition, we also have in \eqref{semi2D1} that
\begin{align*}
\mathbf{A}= \sum_{i=1}^{n_y} A^{a_{:,i}}\otimes E_y^i H_y + \sum_{j=1}^{n_x} E_x^j H_x\otimes A^{b_{j,:}},
\end{align*} 
where $A^{a_{:,i}}$ is the symmetric semidefinite matrix associated with $D^{a_{:,i}}$. The right-hand side of \eqref{semi2D1} are SAT imposing the Dirichlet boundary conditions. We define
\begin{align*}
d_W &= \sum_{i=1}^{n_y} d^{a_{1,i}}\otimes E_y^i,\ d_E=\sum_{i=1}^{n_y} d^{a_{nx,i}}\otimes E_y^i,\\
d_S&=\sum_{j=1}^{n_x}E_x^j\otimes d^{b_{j,1}},\ d_N  = \sum_{j=1}^{n_x} E_x^j\otimes d^{b_{j,ny}},
\end{align*}
where $d^{a_{1,i}}$ approximates the first derivative $a(x_1,y_i)\f{\p}{\p x}$ and is associated with $D^{a_{:,i}}$. The $n_y$ by 1 vector $\bs{f_{tW}}$ is the time derivative of the Dirichlet boundary data evaluated on the grid of the left boundary $x=0$.  We use the following operators to select the numerical solutions on the boundary
\begin{align*}
e_W & = e_{1x}\otimes I_y, \ e_E =  e_{nx}\otimes \reviii{I_y}, \ e_S  = I_x\otimes e_{1y}, \ e_N  = I_x\otimes e_{ny}.
\end{align*}
Finally, the \reviii{third} term on the right-hand side of \eqref{semi2D2} corresponds to numerical dissipation and the parameter $\theta$ is determined by the energy analysis. The grid function $\bs{F}$ is the forcing function $F$ evaluated on the grid. 

\begin{theorem}
The semi-discretization \eqref{semi2D1}-\eqref{semi2D2} satisfies
\[
\f{d}{dt}E_H\leq 0,
\]
if $\theta\leq0$, where the discrete energy is $E_H\equiv \bs{u}^T \mathbf{A} \bs{u}+\bs{v}^T\mathbf{H}\bs{v}$.
\end{theorem}
\begin{proof}
Consider homogeneous boundary and forcing data. 
Multiplying \eqref{semi2D1} by $\bs{u}^T$, we have 
\begin{align}\label{ea2D1}
\bs{u}^T \mathbf{A} \bs{u}_t =&\ \bs{u}^T \mathbf{A} \bs{v} + \bs{u}^T{d_W} H_y e_W^T \bs{v} - \bs{u}^T d_E H_y e_E^T \bs{v} + \bs{u}^T d_S H_x e_S^T \bs{v} - \bs{u}^T d_N H_x  e_N^T \bs{v}.
\end{align}
Similarly, we multiply \eqref{semi2D2} by $\bs{v}^T\mathbf{H}$ and obtain
\begin{align}\label{ea2D2}
\bs{v}^T\mathbf{H}\bs{v}_t =&\ \bs{v}^T\mathbf{H} \mathbf{D}\bs{u}+\theta \bs{v}^T (e_W H_y e_W^T\bs{v}+e_E H_y e_E^T\bs{v}+e_S H_x e_S^T\bs{v}+e_N H_x e_N^T\bs{v}) \notag\\
=&-\bs{v}^T\mathbf{A} \bs{u} -\bs{v}^T e_W H_y d_W^T \bs{u} + \bs{v}^T e_E H_y d_E^T \bs{u} - \bs{v}^T e_S H_x d_S^T \bs{u} + \bs{v}^T e_N H_x d_N^T \bs{u}\notag\\
&\ +\theta \bs{v}^T (e_W H_y e_W^T\bs{v}+e_E H_y e_E^T\bs{v}+e_S H_x e_S^T\bs{v}+e_N H_x e_N^T\bs{v}).
\end{align}
We then add \eqref{ea2D1} and \eqref{ea2D2} to obtain
\begin{align}
\f{d}{dt}(\bs{u}^T \mathbf{A} \bs{u}+\bs{v}^T\mathbf{H}\bs{v}) = &\ 2 \bs{u}^T \mathbf{A} \bs{u}_t+2\bs{v}^T\mathbf{H}\bs{v}_t \notag\\
=&\ 2\theta \bs{v}^T (e_W H_y e_W^T\bs{v}+e_E H_y e_E^T\bs{v}+e_S H_x e_S^T\bs{v}+e_N H_x e_N^T\bs{v}).
\end{align}
Therefore, the discrete energy $E_H\equiv\bs{u}^T \mathbf{A} \bs{u}+\bs{v}^T\mathbf{H}\bs{v}$ decays in time if $\theta<0$ and \reviii{is conserved} in time if $\theta=0$. \qed
\end{proof}


To advance in time the two dimensional semi-discretized equations \eqref{semi2D1}-\eqref{semi2D2}, we need to isolate the ${\bf u}_t$ term. In Section \ref{sec_cc} we show that when the problem has constant coefficients this can be efficiently done through the diagonalization technique \revii{first proposed in \cite{lynch1964direct} and more recently used in \cite{Zhang2020a} for a Galerkin-difference method. }In Section \ref{sec_vc} we illustrate that the variable coefficient case can be handled by the use of iterative solvers. 

\subsection{Constant coefficient problems}\label{sec_cc}
Now consider the case $a(x,y)\equiv a$ and $b(x,y)\equiv b$, where $a,b$ are positive \reviii{constants}. The semi-discretized equation is the same as \eqref{semi2D1}-\eqref{semi2D2} but the operators $\mathbf{D}$, $\mathbf{A}$ and $d_{W,E,S,N}$ are in a simpler form
\begin{align*}
 &\mathbf{D} = a D_x\otimes I_y + b I_x\otimes D_y, \
  \mathbf{A} = a A_x\otimes H_y + b H_x\otimes A_y, \\
  &d_{W} = a\bs{d_{1x}}\otimes I_y,\   d_{E} = a\bs{d_{nx}}\otimes I_y, \ d_S = b I_x\otimes \bs{d_{1y}}, \ d_N = b I_x\otimes \bs{d_{ny}},        
\end{align*}
where $D_x=H_x^{-1}(-A_x-\bs{e_{1x}}\bs{d}^T_{\bs{1x}}+\bs{e_{nx}}\bs{d}^T_{\bs{nx}})$ approximates $\p^2/\p x^2$ and the operators with subscript $y$ are defined analogously. We further define 
\[
\wt{\mathbf{A}}=\mathbf{H}^{-\frac{1}{2}} \mathbf{A} \mathbf{H}^{-\frac{1}{2}}=(\underbrace{aH_x^{-\frac{1}{2}} A_x H_x^{-\frac{1}{2}}}_{\wt{A_x}})\otimes I_y + I_x\otimes (\underbrace{bH_y^{-\frac{1}{2}} A_y H_y^{-\frac{1}{2}}}_{\wt{A_y}})
\]
and consider the eigendecomposition of $\wt{A_x}$ and $\wt{A_y}$,
\[\wt{A_x}Q_x=Q_x\Lambda_x,\quad \wt{A_y}Q_y=Q_y\Lambda_y.\]
Here, $\Lambda_{x}$ is a diagonal matrix with the eigenvalues of $\wt{A_{x}}$ as the diagonal entries. Since $\wt{A_{x}}$ is real and symmetric, the eigenvectors can be chosen to be orthogonal $Q_{x}^T = Q_{x}^{-1}$. The operators $Q_y$ and $\Lambda_y$ are defined analogously. The operator $\wt{\mathbf{A}}$ can be diagonalized as
\[
\mathbf{Q}^T\wt{\mathbf{A}}{\mathbf{Q}}=\mathbf{\Lambda},
\] 
where the orthogonal matrix $\mathbf{Q}=Q_x\otimes Q_y$ and the diagonal matrix $\mathbf{\Lambda}=\Lambda_x\otimes I_y+I_x\otimes\Lambda_y$.

Next, we define $\wt{\bs{u}}$ and $\wt{\bs{v}}$ such that they satisfy 
\begin{align}\label{tildeuv}
{\bs{u}}=\mathbf{H}^{-\frac{1}{2}}\mathbf{Q}\wt{\bs{u}} \text{ and } {\bs{v}}=\mathbf{H}^{-\frac{1}{2}}\mathbf{Q}\wt{\bs{v}},
\end{align}
 respectively. Substituting the new variables into \eqref{semi2D1}, we obtain
\begin{align*}
  \mathbf{A}\mathbf{H}^{-\frac{1}{2}}\mathbf{Q}(\wt{\bs{u}}_t-\wt{\bs{v}}) = &\ {d_W} H_y \left( e_W^T \mathbf{H}^{-\frac{1}{2}}\mathbf{Q}\wt{\bs{v}}- \bs{f_{tW}}\right)  
 - d_E H_y\left(e_E^T \mathbf{H}^{-\frac{1}{2}}\mathbf{Q}\wt{\bs{v}}- \bs{f_{tE}}\right)  \\
 &+ d_S H_x \left( e_S^T \mathbf{H}^{-\frac{1}{2}}\mathbf{Q}\wt{\bs{v}}- \bs{f_{tS}}\right)  
 - d_N H_x \left( e_N^T \mathbf{H}^{-\frac{1}{2}}\mathbf{Q}\wt{\bs{v}}- \bs{f_{tN}}\right).
  \end{align*}
  \reviii{We multiply the above equation from the left by $(\mathbf{H}^{-\frac{1}{2}}\mathbf{Q})^T$, and obtain}  
\begin{align}\
  \mathbf{\Lambda}(\wt{\bs{u}}_t-\wt{\bs{v}}) =&\ (\mathbf{H}^{-\frac{1}{2}}\mathbf{Q})^T \left[ {d_W} H_y \left( e_W^T \mathbf{H}^{-\frac{1}{2}}\mathbf{Q}\wt{\bs{v}}- \bs{f_{tW}}\right)  
 - d_E H_y\left(e_E^T \mathbf{H}^{-\frac{1}{2}}\mathbf{Q}\wt{\bs{v}}- \bs{f_{tE}}\right) \right. \notag\\
 &\left.+ d_S H_x \left( e_S^T \mathbf{H}^{-\frac{1}{2}}\mathbf{Q}\wt{\bs{v}}- \bs{f_{tS}}\right)  
 - d_N H_x \left( e_N^T \mathbf{H}^{-\frac{1}{2}}\mathbf{Q}\wt{\bs{v}}- \bs{f_{tN}}\right)\right].\label{semi2D1t}
  \end{align}
We note that the diagonal matrix $\mathbf{\Lambda}$ has one eigenvalue equal to zero. If we order the eigenvalues such that the first diagonal entry of $\mathbf{\Lambda}$ is zero, \reviii{then the first equation of \eqref{semi2D1t} is $(\wt{u}_1)_t=\wt{v}_1$.} 
  
In the same way, we can substitute the new variables $\wt{\bs{u}}$ and $\wt{\bs{v}}$  into \eqref{semi2D2} and obtain
\begin{align*}
\mathbf{H}^{-\frac{1}{2}}\mathbf{Q}\wt{\bs{v}}_t = &\ \mathbf{D}\mathbf{H}^{-\frac{1}{2}}\mathbf{Q}\wt{\bs{u}}+\bs{F} \\
&+\theta \mathbf{H}^{-1} [e_W H_y (e_W^T\mathbf{H}^{-\frac{1}{2}}\mathbf{Q}\wt{\bs{v}}-\bs{f_{tW}})+e_E H_y (e_E^T\mathbf{H}^{-\frac{1}{2}}\mathbf{Q}\wt{\bs{v}}-\bs{f_{tE}}) \notag\\
&+e_S H_x (e_S^T\mathbf{H}^{-\frac{1}{2}}\mathbf{Q}\wt{\bs{v}}-\bs{f_{tS}})+e_N H_x (e_N^T\mathbf{H}^{-\frac{1}{2}}\mathbf{Q}\wt{\bs{v}}-\bs{f_{tN}})].
\end{align*}
\reviii{We then multiply the above equation from the left by $\mathbf{Q}^T \mathbf{H}^{\frac{1}{2}}$, and have} 
\begin{align}
\wt{\bs{v}}_t = &\ \mathbf{Q}^T \mathbf{H}^{\frac{1}{2}} \mathbf{D}\mathbf{H}^{-\frac{1}{2}}\mathbf{Q}\wt{\bs{u}}+\mathbf{Q}^T \mathbf{H}^{\frac{1}{2}}\bs{F}+\theta \mathbf{Q}^T \mathbf{H}^{-\frac{1}{2}} [e_W H_y (e_W^T\mathbf{H}^{-\frac{1}{2}}\mathbf{Q}\wt{\bs{v}}-\bs{f_{tW}})\label{semi2D2t} \\ 
&+e_E H_y (e_E^T\mathbf{H}^{-\frac{1}{2}}\mathbf{Q}\wt{\bs{v}}-\bs{f_{tE}}) 
+e_S H_x (e_S^T\mathbf{H}^{-\frac{1}{2}}\mathbf{Q}\wt{\bs{v}}-\bs{f_{tS}})+e_N H_x (e_N^T\mathbf{H}^{-\frac{1}{2}}\mathbf{Q}\wt{\bs{v}}-\bs{f_{tN}})].\notag
\end{align}
The transformed difference operator $\mathbf{Q}^T \mathbf{H}^{\frac{1}{2}} \mathbf{D}\mathbf{H}^{-\frac{1}{2}}\mathbf{Q}$ can be simplified by using the relation 
\[
\mathbf{D} = \mathbf{H}^{-1}(-\mathbf{A} -e_W H_y d_W^T + e_E H_y d_E^T-e_S H_x d_S^T + e_N H_x d_N^T) 
\]
to obtain
\begin{align*}
\mathbf{Q}^T \mathbf{H}^{\frac{1}{2}} \mathbf{D}\mathbf{H}^{-\frac{1}{2}}\mathbf{Q} &=   \mathbf{Q}^T \mathbf{H}^{-\frac{1}{2}} (-\mathbf{A} -e_W H_y d_W^T + e_E H_y d_E^T-e_S H_x d_S^T + e_N H_x d_N^T) \mathbf{H}^{-\frac{1}{2}}\mathbf{Q} \\
&= -\mathbf{\Lambda}+ \mathbf{Q}^T \mathbf{H}^{-\frac{1}{2}}  ( -e_W H_y d_W^T + e_E H_y d_E^T-e_S H_x d_S^T + e_N H_x d_N^T) \mathbf{H}^{-\frac{1}{2}}\mathbf{Q}.
\end{align*}
The operator $\mathbf{A}$, which is the volume part of $\mathbf{D}$, is diagonalized to $\mathbf{\Lambda}$. For the boundary parts, we do not need to use the $n_xn_y$-by-$n_xn_y$ dense matrix $\mathbf{Q}$. As an example, for the term  $\mathbf{Q}^T \mathbf{H}^{-\frac{1}{2}} e_W H_y d_W^T \mathbf{H}^{-\frac{1}{2}}\mathbf{Q}$, we have
\begin{align*}
&\mathbf{Q}^T \mathbf{H}^{-\frac{1}{2}} e_W H_y d_W^T \mathbf{H}^{-\frac{1}{2}}\mathbf{Q}\\
=& (Q_x^T \otimes Q_y^T) (H_x^{-\frac{1}{2}}\otimes H_y^{-\frac{1}{2}}) (e_{1x}\otimes I_y)(1\otimes H_y)(ad_{1x}^T\otimes I_y)(H_x^{-\frac{1}{2}}\otimes H_y^{-\frac{1}{2}})(Q_x \otimes Q_y)\\
=&(Q_x^T H_x^{-\frac{1}{2}} e_{1x} ad_{1x}^T H_x^{-\frac{1}{2}}Q_x)\otimes(Q_y^T H_y^{-\frac{1}{2}} H_y H_y^{-\frac{1}{2}} Q_y)\\
=&(Q_x^T H_x^{-\frac{1}{2}} e_{1x} ad_{1x}^T H_x^{-\frac{1}{2}}Q_x)\otimes I_y
\end{align*}
Furthermore, the rank of the $n_x$-by-$n_x$ matrix $e_{1x} ad_{1x}^T$ is 1. Hence, multiplying the $n_xn_y$-by-$n_xn_y$ matrix with the vector $\wt{\mathbf{u}}$ can be done with computational complexity $\mathcal{O}(n_xn_y)$. This procedure can be used for the computation of the other boundary terms in \eqref{semi2D1t} and \eqref{semi2D2t}. 
In the end, the solution to the original problem can be obtained by \eqref{tildeuv}. 

\subsection{Variable coefficient problems}\label{sec_vc}
The above diagonalization procedure cannot be easily generalized to solve \reviii{for}  problems with variable coefficients (either originating from heterogeneous material properties or grid transformation). Instead we may simply solve ${\bf A u}_t$ in each timestep by an iterative method. As  ${\bf A} $ is symmetric and since the right hand side will always be in the range of $A$ the method of choice is the preconditioned conjugate gradient method. Below, in Section \ref{sec:2D_exp} we show that when an incomplete Cholesky preconditioner is used  together with the initial guess      $\bs{u}_t\approx \bs{v}$, the number of iterations needed to meet a tolerance that scales with the order of the method is small.

\section{Numerical experiments}

We present numerical examples in Section \ref{sec:1D_exp} to verify \reviii{the} convergence property of our proposed method. In all experiments, the classical Runge-Kutta method is used for time integration. The $L_2$ errors at final time are computed as 
\[
\|\bs{u_h}-\bs{u_{ex}}\| = \sqrt{h^d (\bs{u_h}-\bs{u_{ex}})^T (\bs{u_h}-\bs{u_{ex}})},
\]
where $\bs{u_h}$ is the numerical solution, $\bs{u_{ex}}$ is the manufactured solution restricted to the grid, $h$ is the grid spacing and $d$ is the spatial dimension. The convergence rates for grids refined by a factor of two are estimated by 
\[
\log_2\f{\|\bs{u_{2h}}-\bs{u_{ex}}\|}{\|\bs{u_h}-\bs{u_{ex}}\|}.
\]
In Section \ref{sec:2D_exp}, we test the preconditioned conjugate gradient method for solving 2D wave equation with \reviii{variable coefficients}.

\subsection{Examples in one space dimension}\label{sec:1D_exp}

We start with a verification of \reviii{the} convergence rate for the wave equation \reviii{$U_{tt}=U_{xx}$} in the domain $x\in [-\pi/2,\pi/2]$ and $t\in [0,2]$. We consider the Dirichlet boundary conditions at $x=-\pi/2$ and $x=\pi/2$. The boundary data is obtained from the manufactured solution $U=\cos(10x+1)\cos(10t+2)$. 

\reviii{We construct the semi-discretization based on  \eqref{wave1dsingle1}-\eqref{wave1dsingle2} by using the SBP operators in \cite{Mattsson2004} of fourth and sixth order of accuracy, compute explicitly the pseudoinverse of $A$.} We are interested in how the dissipative term affects the accuracy of the numerical solution. To this end, we consider the parameter $\beta=0$ or $-1$ to control the dissipation. The $L_2$ errors and the corresponding rates of convergence are presented in Table \ref{tab_4th} for the fourth order method and Table \ref{tab_6th} for the sixth order method.  

\begin{table}[htb]
\begin{center}
\begin{minipage}{.49\linewidth}
  \begin{tabular}{lll}
    \hline
    & $\beta=0$ & \\ \hline
    $n$ & $L_2$ error & rate \\ \hline
    101 & $1.1469\times 10^{-2}$ &  \\ \hline
    201 & $1.5189\times 10^{-3}$ & 2.9166 \\ \hline
    401 & $1.9285\times 10^{-4}$ & 2.9775 \\ \hline
    801 & $2.4215\times 10^{-5}$ & 2.9934 \\ \hline
    1601 & $3.0314\times 10^{-6}$ & 2.9978 \\
    \hline
  \end{tabular}
 \end{minipage}
 \begin{minipage}{.49\linewidth}
  \begin{tabular}{lll}
    \hline
    & $\beta=-1$ & \\ \hline
    $n$ & $L_2$ error & rate \\ \hline
    101 & $5.8872\times 10^{-4}$ &  \\ \hline
    201 & $3.5251\times 10^{-5}$ & 4.0618\\ \hline
    401 & $2.1593\times 10^{-6}$ & 4.0290 \\ \hline
    801 & $1.3419\times 10^{-7}$ & 4.0082 \\ \hline
    1601 & $8.3723\times 10^{-9}$ & 4.0025 \\
    \hline
  \end{tabular}
 \end{minipage}
 \caption{The fourth order SBP-SAT method for the one dimensional wave equation in a single domain.\label{tab_4th}} 
 
 \end{center}
 \end{table}

\begin{table}[htb]
\begin{minipage}{.49\linewidth}
  \begin{tabular}{lll}
    \hline
    & $\beta=0$ & \\ \hline
    $n$ & $L_2$ error & rate \\ \hline
    101 & $3.4741\times 10^{-3}$ &  \\ \hline
    201 & $1.1656\times 10^{-4}$ & 4.8975 \\ \hline
    401 & $3.7103\times 10^{-6}$ & 4.9733 \\ \hline
    801 & $1.1652\times 10^{-7}$ & 4.9929 \\ \hline
    1601 & $3.6466\times 10^{-9}$ & 4.9979 \\
    \hline
  \end{tabular}
 \end{minipage}
 \begin{minipage}{.49\linewidth}
  \begin{tabular}{lll}
    \hline
    & $\beta=-1$ & \\ \hline
    $n$ & $L_2$ error & rate \\ \hline
    101 & $7.4933\times 10^{-5}$ &  \\ \hline
    201 & $1.4300\times 10^{-6}$ & 5.7155\\ \hline
    401 & $2.9257\times 10^{-8}$ & 5.6111 \\ \hline
    801 & $6.1548\times 10^{-10}$ & 5.5709 \\ \hline
    1601 & $1.3250\times 10^{-11}$ & 5.5377 \\
    \hline
  \end{tabular}
 \end{minipage}
 \caption{The sixth order SBP-SAT method for the one dimensional wave equation in a single domain. \label{tab_6th}} 

 \end{table}

We observe that the parameter $\beta$ affects the numerical errors and convergence rates. For the fourth order method, fourth order convergence rate is obtained when $\beta=-1$. However, when $\beta=0$ the convergence rate drops by one order to three. This agrees with the error estimate in Section \ref{sec_nm}. For the sixth order method, the choice $\beta=-1$ leads to a super-convergence of order 5.5. \revii{The same convergence rate is observed and proved in  \cite{Wang2017} for the traditional sixth order SBP-SAT discretization for the Dirichlet problem. With a careful analysis of the solution to the boundary system, it was shown that the coefficient multiplied with the slowly decaying component of the error in Laplace space equals to zero. This leads to an additional gain of a half order in convergence rate.  Without dissipation from the Dirichlet boundary, however, the convergence rate is five.}

\reviii{Next, we test the numerical interface treatment \eqref{semi_u}-\eqref{semi_v_tilde}, and consider the same problem as above but with a grid interface at $x=0$ and interface conditions \eqref{int1}-\eqref{int2}.} To eliminate any influence from the boundaries, we impose periodic boundary condition at $x=\pm\pi/2$. For the interface conditions, we choose either $\gamma=-1$ or $\gamma=0$, corresponding to with or without dissipation at the interface, respectively. The $L_2$ errors and the corresponding rates of convergence are presented in Table \ref{tab_4th_int} for the fourth order method and Table \ref{tab_6th_int} for the sixth order method.  

\begin{table}
\begin{minipage}{.49\linewidth}
  \begin{tabular}{lll}
    \hline
    & $\gamma=0$ & \\ \hline
    $n$ & $L_2$ error & rate \\ \hline
    51 & $1.6233\times 10^{-4}$ &  \\ \hline
    101 & $6.9416\times 10^{-6}$ & 4.5475 \\ \hline
    201 & $3.3128\times 10^{-7}$ & 4.3892 \\ \hline
    401 & $1.8150\times 10^{-8}$ & 4.1900 \\ \hline
    801 & $1.0787\times 10^{-9}$ & 4.0726 \\
    \hline
  \end{tabular}
 \end{minipage}
 \begin{minipage}{.49\linewidth}
  \begin{tabular}{lll}
    \hline
    & $\gamma=-1$ & \\ \hline
    $n$ & $L_2$ error & rate \\ \hline
    51 & $1.2908\times 10^{-4}$ &  \\ \hline
    101 & $6.6070\times 10^{-5}$ & 4.2881\\ \hline
    201 & $3.2790\times 10^{-6}$ & 4.3327 \\ \hline
    401 & $1.8134\times 10^{-7}$ & 4.1764 \\ \hline
    801 & $1.0788\times 10^{-9}$ & 4.0715 \\
    \hline
  \end{tabular}
  \end{minipage}
 \caption{The fourth order SBP-SAT method for the one dimensional wave equation with a grid interface.} 
 \label{tab_4th_int}
 \end{table}

\begin{table}
\begin{minipage}{.49\linewidth}
  \begin{tabular}{lll}
    \hline
    & $\gamma=0$ & \\ \hline
    $n$ & $L_2$ error & rate \\ \hline
    51 & $9.4638\times 10^{-5}$ &  \\ \hline
    101 & $1.4000\times 10^{-6}$ & 6.0790 \\ \hline
    201 & $3.1396\times 10^{-8}$ & 5.4786 \\ \hline
    401 & $8.1443\times 10^{-10}$ & 5.2686 \\ \hline
    801 & $2.2536\times 10^{-11}$ & 5.1755 \\
    \hline
  \end{tabular}
 \end{minipage}
 \begin{minipage}{.49\linewidth}
  \begin{tabular}{lll}
    \hline
    & $\gamma=-1$ & \\ \hline
    $n$ & $L_2$ error & rate \\ \hline
    51 & $5.0107\times 10^{-5}$ &  \\ \hline
    101 & $1.2083\times 10^{-6}$ & 5.3739\\ \hline
    201 & $2.6278\times 10^{-8}$ & 5.5230 \\ \hline
    401 & $5.7619\times 10^{-10}$ & 5.5112 \\ \hline
    801 & $1.2723\times 10^{-11}$ & 5.5010 \\
    \hline
  \end{tabular}
  \end{minipage}
 \caption{The sixth order SBP-SAT method for the one dimensional wave equation with a grid interface.} 
 \label{tab_6th_int}
 \end{table}

For the fourth order method, both $\gamma=0$ and $\gamma=-1$ lead to a convergence rate of order four. For the same mesh resolution, the L$_2$ errors are almost the same. For the sixth order method, the two choices of $\gamma$ give different rates of convergence. The convergence rate with $\gamma=-1$ is 5.5, but the rate drops to between 5 and 5.5 when $\gamma=0$. On the finest mesh, the L$_2$ error with the dissipative discretization is about half of the L$_2$ error with the energy-conserving discretization. \revii{We note that the with the traditional sixth order SBP-SAT discretization for the interface problem, the convergence rate is 5.5 \cite{Wang2017}, which is the same as the dissipative energy-based SBP-SAT discretization. }

\subsection{Examples in two space dimensions} \label{sec:2D_exp}

We consider variable coeffcient problem \eqref{wave2dsys} with $a(x,y)=0.5(\tanh(k(R-0.25))+3)$ and $b(x,y)=0.5(\tanh(k(R-0.25))+3)$, where $R=(x-0.5)^2+(y-0.5)^2$. The parameters $a$ and $b$ model heterogeneous material properties of a layered structure, with $k$ controlling the transition of two layers. A few examples of $a(x,y)$ with different values of $k$ are shown in Figure \ref{Plota}. We see that with a larger $k$, the transition zone of \reviii{the} two materials becomes smaller. We choose the forcing function $F$ so that the manufactured solution $U=\cos(2x+\pi/2)\cos(2y+\pi/2)\cos(2\sqrt{2}t+3)$ satisfies the equations.

\begin{figure}
     \begin{subfigure}[b]{0.45\textwidth}
\includegraphics[width=\textwidth]{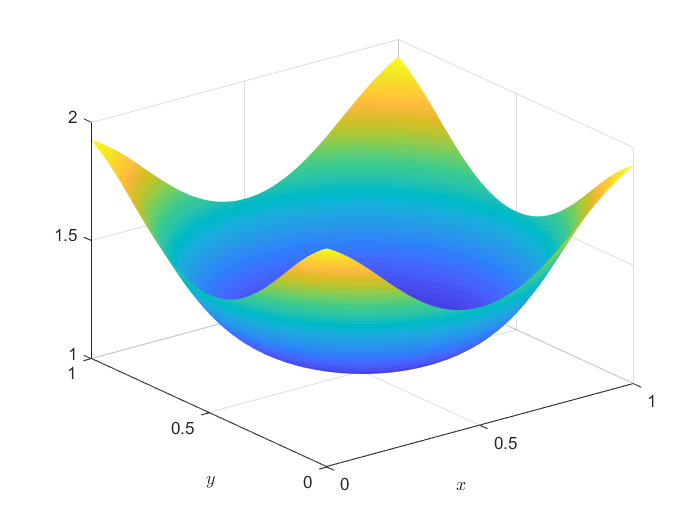}
\caption{$k=5$}
\end{subfigure}
     \begin{subfigure}[b]{0.45\textwidth}
\includegraphics[width=\textwidth]{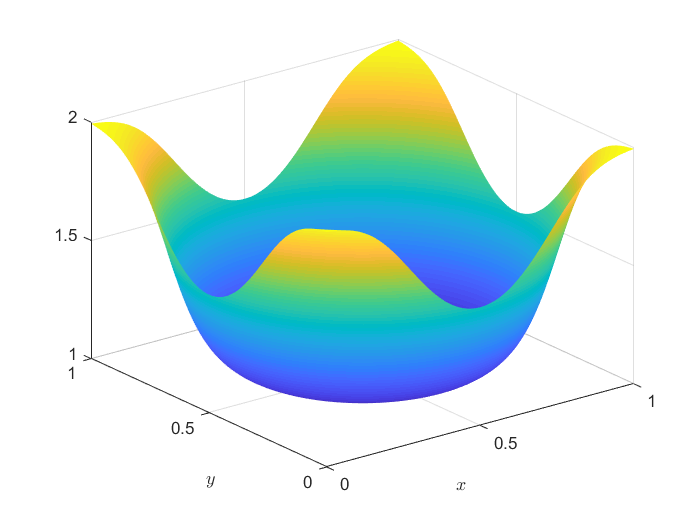}
\caption{$k=10$}
\end{subfigure}
     \\
     \begin{subfigure}[b]{0.45\textwidth}
\includegraphics[width=\textwidth]{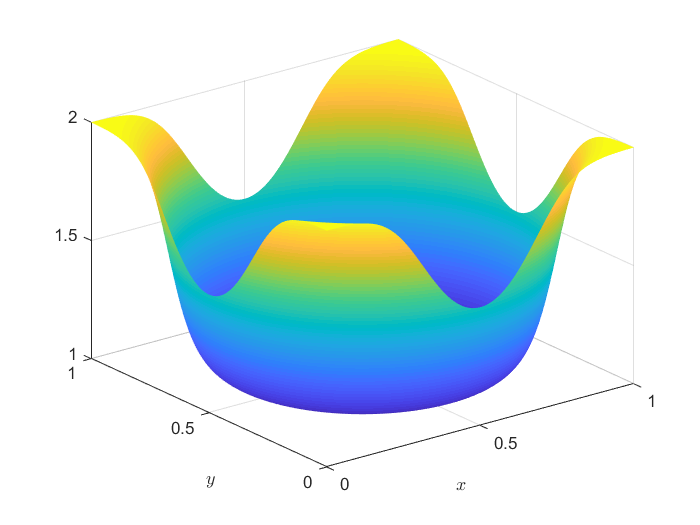}
\caption{$k=15$}
\end{subfigure}
     \begin{subfigure}[b]{0.45\textwidth}
     \includegraphics[width=\textwidth]{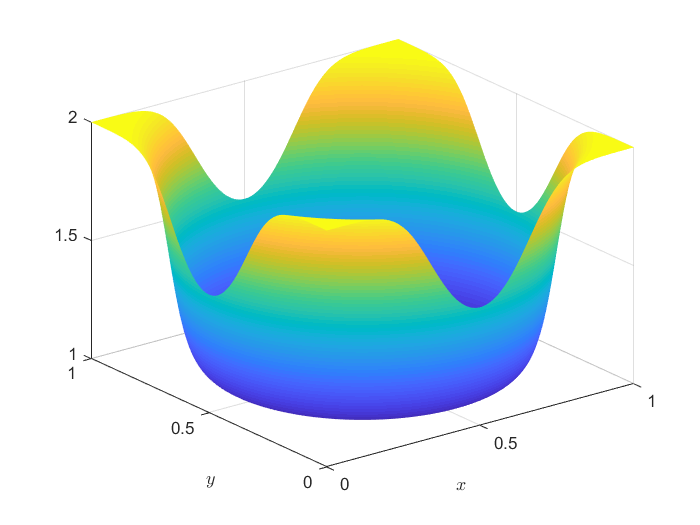}
\caption{$k=20$}
\end{subfigure}
\caption{Material property}
\label{Plota}
\end{figure}

The equations are discretized in space by the scheme  \eqref{semi2D1}-\eqref{semi2D2} with the second derivative variable coefficients SBP operators constructed in \cite{Mattsson2012}. The classical Runge-Kutta method is used to advance the semi-discretization in time. At each Runge-Kutta stage, the linear system is solved by a preconditioned conjugate gradient (PCG) method, where the preconditioner is obtained by the incomplete Cholesky (ICHOL) factorization. Since the matrix $\mathbf{A}$ is only semi-definite, in the ICHOL process we increase the diagonal elements by $1\%$ and $0.01\%$ for the fourth and sixth order methods, respectively. In addition, to keep the factorized matrix sparse, we use a drop-tolerance $10^{-4}$ and $10^{-6}$  for the fourth and sixth order methods, respectively.  In Table \ref{table_nr_iter}, the number of iterations (averaged \reviii{over} the four Runge-Kutta stages in all time steps) are shown for different material properties and mesh resolutions. We observe that in all cases we have tested, PCG converges with less than four iterations.  

\begin{table}
\begin{minipage}{.49\linewidth}
  \begin{tabular}{lllll}
\hline 
 &&Fourth &order&\\
    \hline
    $n$ & $k=5$ & $k=10$ & $k=15$ & $k=20$  \\ \hline
$16^2$ & 2.6& 1.9& 1.3& 1.1\\ \hline
$31^2$ & 3.1 & 2.3& 1.6& 1.3\\ \hline
$61^2$ & 3.3 &2.6 &1.8 & 1.4\\ \hline
$121^2$ &3.5  & 2.7 &2.0& 1.5 \\ \hline
  \end{tabular}
 \end{minipage}
 \begin{minipage}{.49\linewidth}
   \begin{tabular}{lllll}
\hline 
 &&Sixth &order&\\
    \hline
    $n$ & $k=5$ & $k=10$ & $k=15$ & $k=20$  \\ \hline
$16^2$ & 1.0& 1.0& 1.0& 1.0 \\ \hline
$31^2$ & 1.1& 1.2& 1.3& 1.4\\ \hline
$61^2$ & 1.4& 1.5& 1.8& 1.9\\ \hline
$121^2$ &2.2& 2.2& 2.5& 2.8 \\ \hline
  \end{tabular}
  \end{minipage}
 \caption{Number of iterations for solving the linear system: the fourth (left) and sixth (right) order SBP-SAT method for the two dimensional wave equation.} 
\label{table_nr_iter}
 \end{table}

\section{Conclusion}
We have developed an energy-based SBP-SAT discretization of the wave equation.  Comparing with the traditional SBP-SAT discretization, an advantage of the proposed method is that no mesh-dependent parameter is needed to imposed Dirichlet boundary conditions and material interface conditions. Our stability analysis shows that the discretization can either be energy conserving or dissipative. In addition, we have presented a general framework for deriving error estimates by the normal mode analysis and detailed the accuracy analysis for a fourth order discretization. 

In numerical experiments, we have examined more cases for the effect of dissipation on the convergence rate. \reviii{For the fourth order method with Dirichlet boundary conditions, the energy conserving discretization converges to third order, and the dissipative version converges to fourth order.} This is also theoretically proved; while at a grid interface, both dissipative and energy-conserving interface coupling lead to a fourth order convergence rate. For the sixth order method, dissipation at a Dirichlet boundary increases convergence rate from 5 to 5.5. At a grid interface, similar improvement is also observed. 
 
For the energy-based discretization to be efficient the vector ${\bf u}_t$ must be isolated and we have demonstrated that this is possible. For problems in one space dimension, this can be done by explicitly or implicitly forming the pseudoinverse and extracting a few of its columns.  For problems in multiple dimensions with constant coefficients, we have leveraged the diagonalization technique from \cite{Zhang2020a}. This technique gives an algorithm with the same cost as a traditional method of lines discretization  after a pre-computation step that only involve solving one dimensional eigenvalue problems. The same procedure cannot be generalized to problems with variable coefficients. However, our numerical experiments have demonstrated that the corresponding linear system can be solved efficiently by the conjugate gradient method with an incomplete Cholesky preconditioner. The iterative solver converges fast and is \reviii{not very sensitive to the material property and mesh resolution. Here, we have only considered two dimensional problems and observe the time-to-solution for our method and the traditional SBP discretization of the wave equation are roughly comparable. In three dimensions, the preconditioning and iterative solution for  obtaining ${\bf u}_t$ may be less efficient and it remains to be explored if the method presented here can be competitive with the traditional approach.}

\bibliography{Siyang_References}
\bibliographystyle{plain}

\end{document}